\documentclass[12pt,a4paper,oneside]{amsart}

\usepackage{geometry}
\usepackage[sc]{mathpazo}
\usepackage[T1]{fontenc}
\usepackage{amsmath,amssymb,mathtools}
\DeclareMathOperator{\SL}{SL}
\DeclareMathOperator{\E}{E}
\DeclareMathOperator{\GL}{GL}
\DeclareMathOperator{\GE}{GE}
\usepackage{parskip}
\usepackage{etoolbox}
\usepackage{microtype}
\usepackage{xcolor}
\usepackage{placeins}
\usepackage{tikz}
\usetikzlibrary{patterns,patterns.meta}

\definecolor{WrightGreen}{RGB}{34,139,34}
\definecolor{WrightYellow}{RGB}{218,165,32}
\definecolor{linkviolet}{RGB}{112,60,160}
\definecolor{citeblue}{RGB}{35,82,145}

\usepackage{hyperref}
\hypersetup{
  colorlinks=true,
  linkcolor=linkviolet,
  citecolor=citeblue,
  urlcolor=citeblue,
  pdftitle={An explicit non-elementary matrix over a bivariate Laurent polynomial ring},
  pdfauthor={Vasiliy Ionin and Artem Semidetnov},
  pdfsubject={Elementary generation of special linear groups over Laurent polynomial rings}
}

\patchcmd{\section}{\scshape}{\large\bfseries}{}{}
\makeatletter
\renewcommand{\@secnumfont}{\bfseries}
\makeatother
\patchcmd{\thebibliography}{\section*}{\paragraph}{}{}
\numberwithin{equation}{section}
\makeatletter
\let\c@subsection\c@equation

\makeatother
\newtheorem{theorem}{Theorem}

\newtheorem{proposition}[equation]{Proposition}
\newtheorem{lemma}[equation]{Lemma}
\newtheorem*{classicaltheorem}{Theorem}

\theoremstyle{definition}
\newtheorem*{question}{Question}

\title[An explicit non-elementary matrix]{An explicit non-elementary matrix\\
over a bivariate Laurent polynomial ring}

\author{Vasiliy Ionin}
\address{Saint Petersburg Department of Steklov Mathematical Institute\newline
27 Fontanka, St.~Petersburg, 191023, Russia}
\email{ionin.code@gmail.com}

\author{Artem Semidetnov}
\address{Universit\'e de Gen\`eve, Section de Math\'ematiques\\
Route de Drize 7, Villa Battelle, 1227 Carouge, Switzerland}
\email{artemsemidetnov@gmail.com, Artem.Semidetnov@etu.unige.ch}

\date{}

\begin{document}

\begin{abstract}
Let $F$ be the fraction field of a discrete valuation ring.
We give an explicit matrix in
$\SL_2(F[X^{\pm1},Y^{\pm1}])\setminus \E_2(F[X^{\pm1},Y^{\pm1}])$.
The proof of non-elementarity is based on techniques developed by Peter Abramenko
in his preprint \cite{Abramenko}.
\end{abstract}

\maketitle

\tableofcontents

\section{Introduction}

For a commutative ring $R$ and $N\geq 2$, let $\E_N(R)$ be the subgroup of
$\SL_N(R)$ generated by the elementary matrices
$e_{ij}(r):=I+re_{ij}$, where $r\in R$ and $i\neq j$.  Let $\GE_2(R)$ be
the subgroup of $\GL_2(R)$ generated by $\E_2(R)$ and the invertible
diagonal matrices.  The ring $R$ is called a $\GE_2$-ring if
$\GE_2(R)=\GL_2(R)$; for commutative $R$, this is equivalent to
$\SL_2(R)=\E_2(R)$.

Consider the ring
\[
 R=F[X_1^{\pm1},\ldots,X_n^{\pm1},Y_1,\ldots,Y_m],
\]
where $F$ is a field.  It is natural to ask when $\E_N(R)=\SL_N(R)$.
Figure~\ref{fig:landscape} summarizes the known results in rank~$N=2$ as $(n,m)$
varies.

\begin{figure}[ht]
\centering
\begin{tikzpicture}[x=0.72cm,y=0.72cm,every node/.style={font=\scriptsize}]
  \foreach \x in {0,...,5}{
    \foreach \y in {0,...,5}{
      \def\myflag{0}
      \ifnum\x=1 \ifnum\y=1
        \draw[pattern=horizontal lines,pattern color=WrightYellow]
          (\x,\y) rectangle ++(1,1);
        \def\myflag{1}
      \fi\fi
      \ifnum\myflag=0
        \ifnum\x=2 \ifnum\y=0
          \draw[violet,line width=0.55pt] (\x+0.5,\y+0.5) circle (0.3);
          \draw[violet,line width=0.55pt] (\x+0.1,\y+0.5)--(\x+0.9,\y+0.5);
          \draw[violet,line width=0.55pt] (\x+0.5,\y+0.1)--(\x+0.5,\y+0.9);
          \fill[violet] (\x+0.5,\y+0.5) circle (0.045);
          \def\myflag{1}
        \fi\fi
      \fi
      \ifnum\myflag=0
        \def\bmcase{0}
        \ifnum\x>2 \def\bmcase{1}\fi
        \ifnum\x=2 \ifnum\y>0 \def\bmcase{1}\fi\fi
        \ifnum\bmcase=1
          \draw[pattern=north east lines,pattern color=blue]
            (\x,\y) rectangle ++(1,1);
        \fi
        \ifnum\y>1
          \draw[pattern=north west lines,pattern color=red]
            (\x,\y) rectangle ++(1,1);
        \fi
      \fi
      \draw (\x,\y) rectangle ++(1,1);
    }
  }
  \foreach \x/\y in {0/0,1/0,0/1}{
    \draw[WrightGreen!80!black,line width=0.8pt,line cap=round,line join=round]
      (\x+0.22,\y+0.5)--(\x+0.42,\y+0.3)--(\x+0.78,\y+0.72);
  }
  \foreach \y in {0,...,5}{\node at (6.35,\y+0.5) {$\ldots$};}
  \foreach \x in {0,...,5}{\node at (\x+0.5,6.55) {$\vdots$};}
  \node at (6.35,6.55) {$\ddots$};
  \foreach \x in {0,...,5}{\node at (\x+0.5,-0.35) {$\x$};}
  \foreach \y in {0,...,5}{\node at (-0.35,\y+0.5) {$\y$};}
  \node at (3,-0.8) {$n$};
  \node at (-0.8,3) {$m$};
  \begin{scope}[shift={(7.5,0)}]
    \draw (0,5.35) rectangle ++(0.7,0.7);
    \draw[WrightGreen!80!black,line width=0.8pt,line cap=round,line join=round]
      (0.14,5.69)--(0.29,5.54)--(0.57,5.88);
    \node[anchor=west] at (0.9,5.7) {Yes (trivial)};
    \draw[pattern=north west lines,pattern color=red] (0,4.35) rectangle ++(0.7,0.7);
    \node[anchor=west] at (0.9,4.7)
      {No (\hyperlink{cite.Cohn}{Cohn, 1966})};
    \draw[pattern=horizontal lines,pattern color=WrightYellow]
      (0,3.35) rectangle ++(0.7,0.7);
    \node[anchor=west] at (0.9,3.7)
      {No (\hyperlink{cite.Wright}{Wright, 1978})};
    \draw[pattern=north east lines,pattern color=blue] (0,2.35) rectangle ++(0.7,0.7);
    \node[anchor=west] at (0.9,2.7)
      {No (\hyperlink{cite.BachmuthMochizuki}{Bachmuth--Mochizuki, 1982})};
    \draw (0,1.35) rectangle ++(0.7,0.7);
    \draw[violet,line width=0.55pt] (0.35,1.7) circle (0.22);
    \draw[violet,line width=0.55pt] (0.05,1.7)--(0.65,1.7);
    \draw[violet,line width=0.55pt] (0.35,1.4)--(0.35,2.0);
    \fill[violet] (0.35,1.7) circle (0.035);
    \node[anchor=west,align=left] at (0.9,1.7)
      {No when $F$ is the fraction field\\
       of a DVR (Theorem~\ref{thm:main})};
  \end{scope}
\end{tikzpicture}
\caption{Is $\E_2(R)=\SL_2(R)$ for
$R=F[X_1^{\pm1},\ldots,X_n^{\pm1},Y_1,\ldots,Y_m]$?}
\label{fig:landscape}
\end{figure}
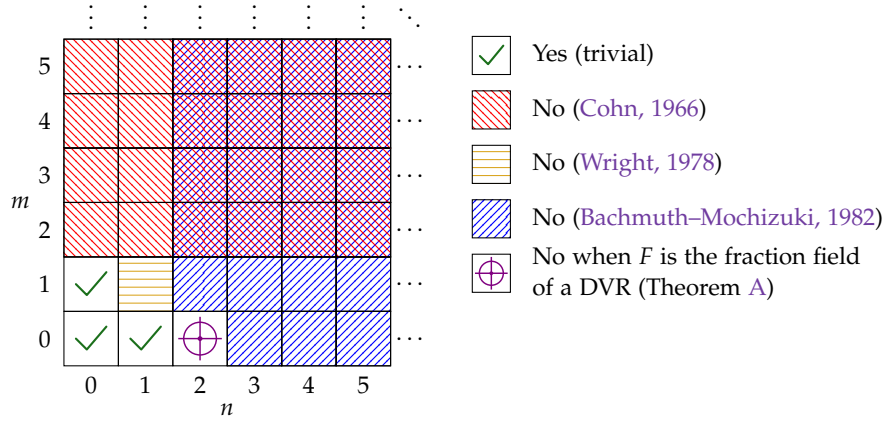

The results underlying the entries of Figure~\ref{fig:landscape}, together
with their historical context, are discussed in
Section~\ref{sec:historical-overview}.  Before the present theorem, these
results left open the case $(n,m)=(2,0)$, which led to the following question.

\begin{question}[{\cite[Problem~2, Question~2]{BanffSession}}]
For which fields $F$, if any, is $F[X^{\pm1},Y^{\pm1}]$ a $\GE_2$-ring?
\end{question}

Our main result answers this question negatively when $F$ is the fraction
field of a discrete valuation ring and provides an explicit non-elementary
matrix.

\begin{theorem}\label{thm:main}
Let $A$ be a discrete valuation ring with fraction field $F$ and uniformizer~$\pi$.
Then
\[
 M_\pi=
 \begin{pmatrix}
 1-\tfrac{1}{\pi}(X-1)(Y-1) & \tfrac{1}{\pi}(Y-1)^2\\[6pt]
 -\tfrac{1}{\pi}(X-1)^2 & 1+\tfrac{1}{\pi}(X-1)(Y-1)
 \end{pmatrix}
 \in \SL_2(F[X^{\pm1},Y^{\pm1}])
\]
does not belong to $\E_2(F[X^{\pm1},Y^{\pm1}])$.
\end{theorem}

The proof, given in Section~\ref{sec:proof}, uses Serre's classical amalgam
decomposition and two ingredients from Abramenko \cite{Abramenko}.  We include
a variant of \cite[Lemma~2.6]{Abramenko} and the required special case of
\cite[Lemma~3.2(ii)]{Abramenko}.

Equivalently, Theorem~\ref{thm:main} applies whenever $F$ admits a nontrivial
discrete valuation; examples include $\mathbb{Q}$, $\mathbb{Q}_p$,
$\mathbb{F}_p(t)$, and $\mathbb{F}_p((t))$.  No field for which
$F^\times/\operatorname{Tor}(F^\times)$ is divisible admits a nontrivial
discrete valuation; consequently, the theorem does not apply to any
algebraically closed field, to $\mathbb{R}$, or to any field algebraic over
$\mathbb{F}_p$.  To the best of our knowledge, the cases
$F=\mathbb{F}_p$, $\overline{\mathbb{F}}_p$, $\overline{\mathbb{Q}}$, and
$\mathbb{C}$ remain open.

For example, taking $A=\mathbb{Z}_{(p)}$ in Theorem~\ref{thm:main} gives, for
every prime $p$,
\[
 M_p\in
 \SL_2(\mathbb{Q}[X^{\pm1},Y^{\pm1}])\setminus
 \E_2(\mathbb{Q}[X^{\pm1},Y^{\pm1}]).
\]

\subsection{Novelty}\label{subsec:novelty}

While preparing this paper, we observed that Abramenko's general criterion
\cite[Theorem~1.3]{Abramenko} already implies
\[
 \SL_2(F[X^{\pm1},Y^{\pm1}])\neq \E_2(F[X^{\pm1},Y^{\pm1}])
\]
when $F$ is the fraction field of a discrete valuation ring $A$: apply the
criterion to $A[X^{\pm1},Y^{\pm1}]$ and $\pi$, noting that $\pi$ remains prime,
\[
 \bigcap_{n\geq0}\pi^nA[X^{\pm1},Y^{\pm1}]=0,
\]
and $(A/\pi A)[X^{\pm1},Y^{\pm1}]$ is not a B\'ezout domain because
$(X-1,Y-1)$ is not principal.

This consequence appears not to have been recorded in the literature:
Abramenko posed the question at the 2022 Banff problem session
\cite[Problem~2, Question~2]{BanffSession}, and Stavrova later stated that
equality is not known over an arbitrary field \cite[p.~1]{Stavrova}.  For
$X-1,Y-1$, the proof of \cite[Theorem~3.4]{Abramenko} naturally suggests the
candidate $M_\pi$, but does not prove that this particular matrix is
non-elementary.  Indeed, \cite[Lemma~2.5]{Abramenko}, which underlies
\cite[Theorem~1.3]{Abramenko}, is nonconstructive: it establishes
non-elementarity only for an unspecified matrix in the construction.  Our
contribution is a direct verification of the required coset condition for
$a=e_{21}((X-1)/(Y-1))$; see
Lemma~\ref{lem:coset}.  This makes \cite[Lemma~2.6]{Abramenko} applicable to
$M_\pi$ itself.

\subsection{Acknowledgements}

We warmly thank Anastasia Stavrova for bringing this problem to our attention
and for helpful discussions that significantly influenced this work.

\subsection{AI use statement}

During the initial exploratory stage of this project, we used Claude Opus~5
to search for an analogue of Park's criterion for
$\mathbb{F}_2[X^{\pm1},Y^{\pm1}]$ \cite{Park}.  This attempt was unsuccessful.
We then carried out a computer search, which produced our first candidate
matrix but did not determine computationally whether this matrix belonged to
$\E_2$.  We provided the candidate and the intermediate calculations to Codex
running GPT-5.6 Sol, which found a proof
using the Bruhat--Tits tree associated with a discrete valuation and drew our
attention to Peter Abramenko's work.  The authors subsequently
checked and developed the argument presented here and take responsibility for
all mathematical claims in the paper.

\section{Historical overview}\label{sec:historical-overview}

The first fundamental negative result is Cohn's criterion.

\begin{classicaltheorem}[Cohn, {\cite[Proposition~7.3]{Cohn}}]
Let $R$ be a $k$-ring\footnote{A $k$-ring is a ring equipped with a specified
homomorphism $k\to R$.  For commutative $R$, this notion coincides with that
of a $k$-algebra.} with a degree function, and suppose that $R$ is also a
$\GE_2$-ring.  Then, of any two elements of the same degree which form a
regular row, each is $R$-dependent on the other.
\end{classicaltheorem}

Applied to $F[X,Y]$ with total degree, the row $(1+XY,X^2)$ is regular, its
entries have the same degree, and neither is $F[X,Y]$-dependent on the other.
Consequently,
\[
    \begin{pmatrix}
        1+XY & X^2\\
        -Y^2 & 1-XY
    \end{pmatrix}
    \in
    \SL_2(F[X,Y])\setminus \E_2(F[X,Y]),
\]
and the same argument gives
\[
    \SL_2(F[Y_1,\ldots,Y_m])
    \neq
    \E_2(F[Y_1,\ldots,Y_m])
    \qquad (m\geq2).
\]

In higher rank, Suslin proved a stability theorem for both polynomial and
Laurent polynomial rings; we state it in its original generality.

\begin{classicaltheorem}[Suslin, {\cite[Corollary~7.10]{Suslin}}]
Let $A$ be a regular ring such that
$\operatorname{SK}_1(A)=\SL(A)/\E(A)=0$,\footnote{For commutative $A$, this
group also identifies with
$\ker(\det\colon K_1(A)\to A^\times)$.} and set
\[
    B=
    A[X_1^{\pm1},\ldots,X_k^{\pm1},
      X_{k+1},\ldots,X_s].
\]
Then, for $N\geq \max\{3,\dim A+2\}$, one has
\[
    \SL_N(B)=\E_N(B).
\]
\end{classicaltheorem}

For $A=F$, this gives $\SL_N(R)=\E_N(R)$ for every $N\geq3$, so only
rank~$2$ remains.

In 1978 Wright obtained a structural description of $\GL_2$ over polynomial
rings over a Euclidean domain.  Let $\E_2^+(R)$ denote the subgroup of
$\GL_2(R)$ generated by $\E_2(R)$ and the invertible diagonal matrices.

\begin{classicaltheorem}[Wright, {\cite[Theorem~3]{WrightAnnouncement}}]
Let $R=K[X_1,\ldots,X_s]$, where $K$ is a Euclidean domain.  Then
$\GL_2(R)$ is the free product of $\E_2^+(R)$ with a subgroup $W$,
amalgamated along
\[
    \E_2^+(R)\cap W=:B_2(R).
\]
Moreover, the inclusion $B_2(R)\subseteq W$ is strict unless $R$ is a
Euclidean domain.
\end{classicaltheorem}

If $D$ is a Euclidean domain which is not a field, then $D[X]$ is not
Euclidean, and Wright's theorem gives
\[
    \SL_2(D[X])\neq\E_2(D[X]).
\]
Taking $D=F[X^{\pm1}]$ yields
\[
    \SL_2(F[X^{\pm1},Y])
    \neq
    \E_2(F[X^{\pm1},Y]).
\]

In 1982 Bachmuth and Mochizuki proved a stronger negative result for Laurent
polynomial rings.

\begin{classicaltheorem}[Bachmuth--Mochizuki,
  {\cite[Theorem~1]{BachmuthMochizuki}}]
Let $D$ be an integral domain which is not a field, and put
$R=D[X^{\pm1},Y^{\pm1}]$.  Then any set of generators of $\SL_2(R)$ contains
infinitely many elements outside $\E_2(R)$.
\end{classicaltheorem}

Thus $\SL_2(R)\neq\E_2(R)$.  Applied with
\[
    D=F[X_3^{\pm1},\ldots,X_n^{\pm1},Y_1,\ldots,Y_m],
\]
the theorem settles all cases with $n\geq2$ except $(n,m)=(2,0)$, since
$D$ is not a field whenever $n\geq3$ or $m\geq1$.  Together with Cohn's
and Wright's results, this leaves only the bivariate Laurent polynomial ring
$F[X^{\pm1},Y^{\pm1}]$; Bachmuth and Mochizuki explicitly singled out this
case as open.

Two years later Chu gave a general criterion for the $\GE_2$ problem for
graded rings.

\begin{classicaltheorem}[Chu, {\cite[Theorem~1.2]{Chu}}]
Let $R=\bigoplus_{\alpha\in M}R_\alpha$ be a graded domain of type $M$, where
$M$ is a nonzero totally ordered cancellative commutative monoid.  Assume that
$R_\alpha\neq0$ for every $\alpha\in M$.  If there exist two homogeneous
elements $x,y\in R$ such that
\[
    Rx+Ry
\]
is not a principal ideal, then $R$ is not a $\GE_2$-ring.
\end{classicaltheorem}

In particular, $D[T,T^{-1}]$ is not a $\GE_2$-ring when $D$ is an integral
domain which is not B\'ezout \cite[Example~1.4]{Chu}.  This does not settle
$F[X^{\pm1},Y^{\pm1}]$, since $F[X^{\pm1}]$ is a principal ideal domain.
Park later developed an algorithm deciding membership in
$\E_2(R[x_1,\ldots,x_m])$, for $R$ Euclidean, and producing an elementary
factorization when one exists; its leading-term criterion rules out Cohn's
matrix \cite{Park}.

\section{The proof of Theorem~\ref{thm:main}}\label{sec:proof}

For the proof, fix $A$, $F$, and $\pi$ as in Theorem~\ref{thm:main}, and put
\[
 R=F[X^{\pm1},Y^{\pm1}],
 \qquad d=\operatorname{diag}(\pi,1).
\]
For every subgroup $G\leq \SL_2(F(X,Y))$, write $G^d=d^{-1}Gd$.
The matrix $M_\pi$ admits the factorization
\begin{equation*}\tag{$*$}\label{eq:factorization}
 M_\pi=
 e_{21}\!\left(\frac{X-1}{Y-1}\right)
 \cdot e_{12}\!\left(\tfrac{1}{\pi}(Y-1)^2\right)
 \cdot e_{21}\!\left(\frac{X-1}{Y-1}\right)^{-1}
\end{equation*}
in $\SL_2(F(X,Y))$.
The following two propositions are the main ingredients of the proof.

\begin{proposition}[{\cite[Lemma~3.2(ii)]{Abramenko} applied to
$A[X^{\pm1},Y^{\pm1}]$ and $\pi$}]\label{prop:elementary-bound}
One has
\[
 \E_2(R)\subseteq
 \left\langle
 \SL_2(A[X^{\pm1},Y^{\pm1}]),
 \SL_2(A[X^{\pm1},Y^{\pm1}])^d
 \right\rangle.
\]
\end{proposition}

Let $v\colon F^\times\to\mathbb{Z}$ be the normalized valuation of $A$.
Extend $v$ to $F(X,Y)$ by the Gauss valuation and set $v(0)=\infty$.
For every nonzero Laurent polynomial,
\[
 v\left(\sum_{(i,j)\in\mathbb{Z}^2}c_{ij}X^iY^j\right)
 =\min_{i,j}v(c_{ij}).
\]
Its \emph{valuation ring} is
\[
 \mathcal{O}=\{f\in F(X,Y):v(f)\geq0\}.
\]
It is a discrete valuation ring with uniformizer $\pi$.

The next proposition follows from the argument of \cite[Lemma~2.6]{Abramenko}.
We repeat it below for completeness.

\begin{proposition}\label{prop:amalgam-criterion}
Suppose that
\[
 a\in \SL_2(\mathcal{O}),
 \qquad
 b\in \SL_2(\mathcal{O})^d,
\]
and the following conditions are satisfied:
\begin{enumerate}
\item $a\notin \SL_2(R)\cdot \SL_2(\mathcal{O})^d$, where
$\cdot$ denotes the setwise product of subsets of $\SL_2(F(X,Y))$;
\item $b\notin \SL_2(\mathcal{O})$.
\end{enumerate}
Then
\[
 aba^{-1}\notin
 \left\langle
 \SL_2(A[X^{\pm1},Y^{\pm1}]),
 \SL_2(A[X^{\pm1},Y^{\pm1}])^d
 \right\rangle.
\]
\end{proposition}

We aim to prove Theorem~\ref{thm:main} assuming these propositions.  Their proofs
are given in Subsections~\ref{subsec:elementary-bound-proof} and
\ref{subsec:amalgam-criterion-proof}.

We first give a schematic outline of the proof.
\begin{enumerate}
\item Lemma~\ref{lem:coset} gives the exclusion required in
Proposition~\ref{prop:amalgam-criterion} for the first factor
$a=e_{21}((X-1)/(Y-1))$ in \eqref{eq:factorization}.
\item For the first two factors $a$ and $b$ in \eqref{eq:factorization},
Lemma~\ref{lem:coset} and the valuation of the upper-right entry of $b$ verify
the two conditions of Proposition~\ref{prop:amalgam-criterion}.  The identity
$aba^{-1}=M_\pi$ then shows that $M_\pi$ does not belong to the
subgroup generated by $\SL_2(A[X^{\pm1},Y^{\pm1}])$ and
$\SL_2(A[X^{\pm1},Y^{\pm1}])^d$.
\item Proposition~\ref{prop:elementary-bound} places $\E_2(R)$ in this subgroup.
Thus $M_\pi\notin \E_2(R)$.
\end{enumerate}

Put
\[
 q=\frac{X-1}{Y-1}\in F(X,Y).
\]

\begin{lemma}\label{lem:coset}
One has
\[
 e_{21}(q)\notin \SL_2(R)\cdot \SL_2(\mathcal{O})^d.
\]
\end{lemma}

\begin{proof}
Suppose otherwise.  Write $e_{21}(q)=hg$, where
$h=(h_{ij})\in \SL_2(R)$ and $g\in \SL_2(\mathcal{O})^d$.  Then
\begin{equation}\label{eq:h-membership}
 e_{21}(-q)h=g^{-1}\in \SL_2(\mathcal{O})^d.
\end{equation}

We claim that there is a matrix
\[
 C=(c_{ij})\in M_2(A[X^{\pm1},Y^{\pm1}])
\]
such that $\det C=\pi$ and
\begin{align*}
 (Y-1)c_{21}&\equiv(X-1)c_{11}
 \qquad\bmod \pi A[X^{\pm1},Y^{\pm1}],\\
 (Y-1)c_{22}&\equiv(X-1)c_{12}
 \qquad\bmod \pi A[X^{\pm1},Y^{\pm1}].
\end{align*}
Direct conjugation gives
\[
 \SL_2(\mathcal{O})^d=
 \left\{
 \begin{pmatrix}
 \mathcal{O} & \pi^{-1}\mathcal{O}\\
 \pi\mathcal{O} & \mathcal{O}
 \end{pmatrix}
 \right\}\cap \SL_2(F(X,Y)).
\]
Combining this description with \eqref{eq:h-membership} gives
\begin{equation}\label{eq:h-entry-inclusions}
\begin{array}{l@{\qquad\qquad}l}
 h_{11}\in\mathcal{O},
 & h_{12}\in\pi^{-1}\mathcal{O},\\
 h_{21}-qh_{11}\in\pi\mathcal{O},
 & h_{22}-qh_{12}\in\mathcal{O}.
\end{array}
\end{equation}
Set
\[
 C=h\operatorname{diag}(1,\pi)
 =\begin{pmatrix}
 h_{11}&\pi h_{12}\\
 h_{21}&\pi h_{22}
 \end{pmatrix}
 =(c_{ij}).
\]
Since $v(q)=0$, these inclusions show that every entry of $C$
lies in $\mathcal{O}$; indeed,
$c_{21}=(h_{21}-qh_{11})+qh_{11}$ and
$c_{22}=\pi(h_{22}-qh_{12})+q(\pi h_{12})$.
Every entry also lies in $R$.  Since
$\mathcal{O}\cap R=A[X^{\pm1},Y^{\pm1}]$ by the explicit computation of the
Gauss valuation in Lemma~\ref{lem:intersections},
\[
 C\in M_2(A[X^{\pm1},Y^{\pm1}]),
 \qquad
 \det C=\pi\det h=\pi.
\]
Since $v(Y-1)=0$, multiplication by $Y-1$ preserves $\pi\mathcal{O}$.
Thus multiplying the lower-row inclusions in \eqref{eq:h-entry-inclusions} by
$Y-1$ and $\pi(Y-1)$, respectively, gives the claimed congruences, since
\[
 \pi\mathcal{O}\cap A[X^{\pm1},Y^{\pm1}]
 =\pi A[X^{\pm1},Y^{\pm1}].
\]
This proves the claim.

It remains to show that the matrix $C$ from the claim cannot exist.  Reduce
its congruences in
\[
 (A/\pi A)[X^{\pm1},Y^{\pm1}].
\]
In this unique factorization domain, $Y-1$ is prime and does not divide $X-1$;
hence $Y-1$ divides $c_{11}$ and $c_{12}$.  Similarly, $X-1$ is prime and does
not divide $Y-1$, so $X-1$ divides $c_{21}$ and $c_{22}$.  Thus
\[
 c_{11},c_{12}\in(\pi,Y-1),
 \qquad
 c_{21},c_{22}\in(\pi,X-1).
\]
Every entry of $C$ therefore belongs to the maximal ideal
\[
 \mathfrak{m}=(\pi,X-1,Y-1)
 \subseteq A[X^{\pm1},Y^{\pm1}],
\]
so $\det C\in\mathfrak{m}^2$.
The evaluation at $X=Y=1$ maps $\mathfrak{m}^2$ into $\pi^2A$.
This contradicts $\det C=\pi$.
\end{proof}

\begin{proof}[Proof of Theorem~\ref{thm:main}]
Put
\[
 a=e_{21}(q),
 \qquad
 b=e_{12}\left(\tfrac{1}{\pi}(Y-1)^2\right),
\]
the first two factors in \eqref{eq:factorization}.  Since
$v(X-1)=v(Y-1)=0$, we have $v(q)=0$, and hence
$a\in \SL_2(\mathcal{O})$.  Moreover,
\[
 b=d^{-1}e_{12}((Y-1)^2)d\in \SL_2(\mathcal{O})^d.
\]
Its upper-right entry has valuation $-1$, so $b\notin \SL_2(\mathcal{O})$;
this is condition~(2) of Proposition~\ref{prop:amalgam-criterion}.
Condition~(1) is Lemma~\ref{lem:coset}.  Since
$aba^{-1}=M_\pi$ by \eqref{eq:factorization},
Proposition~\ref{prop:amalgam-criterion} gives
\[
 M_\pi\notin
 \left\langle
 \SL_2(A[X^{\pm1},Y^{\pm1}]),
 \SL_2(A[X^{\pm1},Y^{\pm1}])^d
 \right\rangle.
\]
This subgroup contains $\E_2(R)$ by
Proposition~\ref{prop:elementary-bound}.  Hence $M_\pi\notin \E_2(R)$.
\end{proof}

\subsection{Proof of Proposition~\ref{prop:elementary-bound}}\label{subsec:elementary-bound-proof}

The proposition is proved by direct matrix calculations.

\begin{proof}[Proof of Proposition~\ref{prop:elementary-bound}]
Let
\[
 E=\left\langle \E_2(A[X^{\pm1},Y^{\pm1}]),
 e_{12}(\pi^{-1})\right\rangle.
\]
The identity
\[
 \bigl(e_{12}(\pi^{-1})e_{21}(-\pi)e_{12}(\pi^{-1})\bigr)
 \bigl(e_{12}(-1)e_{21}(1)e_{12}(-1)\bigr)
 =\operatorname{diag}(\pi^{-1},\pi)
\]
shows that $\operatorname{diag}(\pi^{-1},\pi)\in E$.
For $n\geq 0$ and $f\in A[X^{\pm1},Y^{\pm1}]$,
\begin{align*}
 \operatorname{diag}(\pi^{-1},\pi)^n e_{12}(f)
 \operatorname{diag}(\pi^{-1},\pi)^{-n}&=e_{12}(\pi^{-2n}f),\\
 \operatorname{diag}(\pi^{-1},\pi)^{-n} e_{21}(f)
 \operatorname{diag}(\pi^{-1},\pi)^n&=e_{21}(\pi^{-2n}f).
\end{align*}
Every element of $R$ can be written as $\pi^{-2n}f$ for suitable $n\geq0$ and
$f\in A[X^{\pm1},Y^{\pm1}]$.
Thus $E=\E_2(R)$.
Now
$\E_2(A[X^{\pm1},Y^{\pm1}])\subseteq \SL_2(A[X^{\pm1},Y^{\pm1}])$ and
\[
 e_{12}(\pi^{-1})=d^{-1}e_{12}(1)d
 \in \SL_2(A[X^{\pm1},Y^{\pm1}])^d.
\]
The required inclusion follows.
\end{proof}

\subsection{Proof of Proposition~\ref{prop:amalgam-criterion}}\label{subsec:amalgam-criterion-proof}

The argument uses Serre's standard amalgam decomposition
\cite[Chapter~II, \S1.4, Theorem~6]{Serre}:
\[
 \SL_2(F(X,Y))=
 \SL_2(\mathcal{O})
 *_{\SL_2(\mathcal{O})\cap \SL_2(\mathcal{O})^d}
 \SL_2(\mathcal{O})^d.
\]
(Completeness of $F(X,Y)$ with respect to $v$ is not required; Serre's theorem
applies to an arbitrary discretely valued field.)

\begin{lemma}\label{lem:intersections}
The following equalities hold:
\[
 \mathcal{O}\cap R=A[X^{\pm1},Y^{\pm1}],
 \qquad
 \pi\mathcal{O}\cap A[X^{\pm1},Y^{\pm1}]
 =\pi A[X^{\pm1},Y^{\pm1}].
\]
\end{lemma}

\begin{proof}
Every element of $R$ has a unique finite Laurent expansion.
Its Gauss valuation is nonnegative exactly when all coefficients lie in $A$.
This proves the first equality.
An element of $A[X^{\pm1},Y^{\pm1}]$ has valuation at least $1$ exactly when
all coefficients lie in $\pi A$, which proves the second.
\end{proof}

\begin{lemma}\label{lem:group-intersections}
One has
\begin{align*}
 \SL_2(R)\cap \SL_2(\mathcal{O})
 &=\SL_2(A[X^{\pm1},Y^{\pm1}]),\\
 \SL_2(R)\cap \SL_2(\mathcal{O})^d
 &=\SL_2(A[X^{\pm1},Y^{\pm1}])^d.
\end{align*}
\end{lemma}

\begin{proof}
The first equality follows entrywise from Lemma~\ref{lem:intersections}.
Conjugating it by $d$ gives the second, since $\pi\in F^\times\subseteq R^\times$,
so $d\in \GL_2(R)$ and $\SL_2(R)^d=\SL_2(R)$.
\end{proof}

\begin{lemma}[{cf.\ \cite[Lemma~2.6]{Abramenko}}]\label{lem:amalgam}
Let $G=G_0*_{U}G_1$ and $H\leq G$.
If
\[
 a\in G_0\setminus (H\cdot G_1),
 \qquad
 b\in G_1\setminus U,
\]
then
\[
 aba^{-1}\notin
 \langle H\cap G_0,H\cap G_1\rangle.
\]
\end{lemma}

\noindent(In \cite[Lemma~2.6]{Abramenko}, the additional hypothesis
$aba^{-1}\in H$ is imposed; it is not needed for the conclusion above.)

\begin{proof}
Since $U\subseteq G_1\subseteq H\cdot G_1$, one has $a\notin U$.
Thus $aba^{-1}$ is a reduced word of length $3$ with factors in
$G_0,G_1,G_0$.

Suppose that
$aba^{-1}\in\langle H\cap G_0,H\cap G_1\rangle$.
Reducing a word in these generators and applying the normal-form theorem gives
the reduced expression
\[
 aba^{-1}=z_1z_2z_3,
 \qquad
 z_1,z_3\in H\cap G_0,
 \quad z_2\in H\cap G_1.
\]
By the uniqueness part of the normal-form theorem, corresponding syllables in
two reduced expressions of the same element lie in the same right $U$-coset
in the relevant factor.  Comparing the first syllables gives $z_1U=aU$.
Thus $z_1=au$ for some $u\in U$.

Now $z_1\in H\cap G_0\subseteq H$, while $u^{-1}\in U\subseteq G_1$.
Consequently,
\[
 a=z_1u^{-1}\in H\cdot G_1,
\]
contrary to the hypothesis.
\end{proof}

\begin{proof}[Proof of Proposition~\ref{prop:amalgam-criterion}]
Set
\[
 \begin{gathered}
 G=\SL_2(F(X,Y)),\qquad H=\SL_2(R),\\
 G_0=\SL_2(\mathcal{O}),\qquad G_1=\SL_2(\mathcal{O})^d,\\
 U=G_0\cap G_1.
 \end{gathered}
\]
The Serre decomposition above is $G=G_0*_{U}G_1$.
The hypotheses give
\[
 a\in G_0\setminus (H\cdot G_1),
 \qquad
 b\in G_1\setminus G_0.
\]
Since $U\subseteq G_0$, one has $b\notin U$.
Lemma~\ref{lem:amalgam} yields
\[
 aba^{-1}\notin\langle H\cap G_0,H\cap G_1\rangle.
\]
Lemma~\ref{lem:group-intersections} identifies the latter subgroup with
\[
 \left\langle
 \SL_2(A[X^{\pm1},Y^{\pm1}]),
 \SL_2(A[X^{\pm1},Y^{\pm1}])^d
 \right\rangle.
\]
\end{proof}

\begingroup
\hbadness=2000

\endgroup

\end{document}